\documentclass[12pt, reqno]{amsart}

\usepackage{amsmath, amsthm, amscd, amsfonts, amssymb, graphicx, color}
\usepackage[bookmarksnumbered, colorlinks, plainpages]{hyperref}

\newtheorem{theorem}{Theorem}[section]
\newtheorem{cor}[theorem]{Corollary}
\newtheorem{lem}[theorem]{Lemma}

  \newtheorem{defn}[theorem]{Definition}
  
\newtheorem{exam}[theorem]{Example}
\newtheorem{remark}[theorem]{Remark}
\numberwithin{equation}{section}

\newcommand{\BC}{{\Bbb C}}
\newcommand{\BN}{{\Bbb  N}}
\newcommand{\BR}{{\Bbb  R}}

\newcommand{\BK}{{\Bbb K}}

\newcommand{\Ch}{{\rm Ch}}
\newcommand{\E}{{\mathcal{E}}}

\begin{document}


\title[ A Note on nonlinear Isometries Between Vector-Valued Function spaces] {A Note on nonlinear Isometries Between
 Vector-Valued Function spaces}

\author{Arya Jamshidi and Fereshteh Sady }

\subjclass[2010]{Primary 46J10, 46J20; Secondary 47B33}

\keywords{isometries, vector-valued continuous functions, strong
boundary points, Choquet boundary } \maketitle





\maketitle \vspace*{0.5cm} \noindent


\begin{abstract}
Let $X,Y$ be compact Hausdorff spaces and $E,F$ be Banach spaces
over $\BR$ or $\BC$.  In this paper, we investigate the general
form of surjective (not necessarily linear) isometries $T:A
\longrightarrow B$ between subspaces $A$ and $B$ of  $C(X,E)$ and
$C(Y,F)$, respectively. In the case that $F$ is strictly convex,
it is shown that there exist a subset $Y_0$ of $Y$, a continuous
function  $\Phi: Y_0 \longrightarrow X$ onto the set of strong
boundary points of $A$ and a family $\{V_y\}_{y\in Y_0}$ of
real-linear operators from $E$ to $F$  with $\|V_y\| =1$ such that
\[ Tf(y)-T0(y)=V_y(f(\Phi(y))) \,\,\,\,\,\,\,(f \in A, y \in Y_0).\]
In particular, we get some generalizations of the vector-valued
Banach--Stone theorem  and a generalization of Cambern's result.
 We also give a similar result in the case that $F$ is not
 strictly convex, but its unit sphere contains a maximal convex subset which is singleton.
 \end{abstract}



\section{Introduction }
The study of  isometries between subspaces of continuous functions
originally dates back to the classical Banach--Stone theorem. The
theorem  has various  generalizations (in scalar valued case)
based on different techniques, see for example  \cite{Flem1, jam,
miura1, miura2, tonev, lee}.

For a compact Hausdorff space $X$ and a Banach space $E$, let
$C(X,E)$ be the Banach space of continuous $E$-valued functions on
$X$ endowed with supremum norm $\|\cdot\|_\infty$. A
representation theorem for isometries between $C(X,E)$-spaces was
given in \cite{jerr} by Jerison as follows:

Let $X$ and $Y$ be compact Hausddorff spaces,  $E$ be a strictly
convex Banach space and let $T:C(X,E) \longrightarrow C(Y,E)$ be a
surjective linear isometry. Then there exist a continuous surjection
$\Phi: Y \longrightarrow X$ and a map $t \longrightarrow V_t$ which
is continuous from $Y$ into the space $B(E)$ of all bounded
operators on $E$, endowed with the strong operator topology, such
that $Tf(t) = V_t(f(\Phi(t))$  for all $f\in C(X,E)$ and $t \in
Y$.

This  result has been  generalized by Cambern \cite{cam} for into
linear isometries  and by Font \cite{font1} for certain
vector-valued subspaces of continuous functions. In \cite{Al}
Al-Halees and Fleming relaxed the strict convexity condition on
$E$, and by considering $T$-sets in companion with another
condition on $E$, called condition (P), they obtained more general
results. We also refer the  reader to the nice books
\cite{Flem1,Flem2} including many earlier results.

More recently, surjective isometries between certain subspaces of
vector-valued continuous functions  have been studied in
\cite{bote} and \cite{kawa1}. We should note that the method used
in  these  papers  is based on extreme point technique. By
\cite{bote} (see also \cite{Bot-E}), for a compact Hausdorff space
$X$ and a reflexive real Banach space $F$  whose dual is strictly
convex, if $A$ is a subspace of $C(X,F)$ which separates $X$ in
the sense of \cite[Definition 3.1]{bote}, and $T:A \longrightarrow
A$ is a surjective isometry preserving constant functions, then
there exist a surjective isometry $V:F\longrightarrow F$ and a
homeomorphism $\tau : X \longrightarrow X$ such that
\[ Tf(x)=V(f(\tau(x)))\;\;\;\;\;\;\; (f\in A, x\in X) \]
This  result has been generalized in several directions in
\cite{kawa1}.  We refer one of them which is related to our
results. First we state conditions (S3) and (M) introduced in
\cite{kawa1} for a subspace $A$ of $C(X,E)$ where $X$ is a compact
Hausdorff space and $E$ is a Banach space over $\Bbb C$ or $\Bbb
R$:

(S3) For each $x$ in the Choquet boundary $\Ch(A)$ of $A$, for
each neighborhood $U$ of $x$ and for each $u \in E$ there exists a
function  $f \in A$ such that $\|f  \|_\infty = \|u\|$, $f (x) =
u$ and $f =0$ on $X\backslash U$.

\vspace*{.25cm}

(M)  for each $f\in A$ with $f(x)=0$ and for each $\epsilon>0$, there exist a neighborhood $U$ of $x$ and $f_\epsilon\in A$ such that
$\|f-f_\epsilon\|_\infty<\epsilon$ and  $f_\epsilon=0$ on $U$.

\vspace*{.25cm}

\begin{theorem}\cite[Theorem 3.4]{kawa1} \label{kawa}
Let $X$ and $Y$ be compact Hausdorff spaces and let $E$ be a
strictly convex reflexive real or complex Banach space. Assume
that $A$ and $B$ are subspaces of $C(X,E)$, respectively,
containing constant functions and both satisfy  conditions {\rm
(S3)} and {\rm (M)}. Let $T: A \longrightarrow B$ be a surjective
linear isometry. Then there exist a continuous surjection
$\varphi: \Ch(B) \longrightarrow \Ch(A)$ between the Choquet boundaries, and a family $V_y:
E\longrightarrow E$, $y\in \Ch(B)$, of linear operators with
$\|V_y\|=1$ such that
\[ Tf(y)=V_y(f(\varphi(y))) \;\;\;\;\;\;(f\in A, y\in \Ch(B)).\]
If, furthermore, the dual space $E^*$ is strictly convex, then $\varphi$ is a homeomorphism and $V_y$ is an isometric
isomorphism for each $y\in \Ch(B)$.
\end{theorem}

The purpose of this paper is to  study surjective, not necessarily
linear,  isometries  $T : A \longrightarrow B$  between  subspaces
$A$ and $B$ of $C(X, E)$ and $C(Y, F)$, respectively, where $X, Y$
are compact Hausdorff spaces and $E, F$  are Banach spaces (over
$\BR$ or $\BC$). We first assume that $F$ is strictly convex and
give a description of $T$
 on appropriate subset $Y_0$ of $Y$. The given description deals with the set of strong boundary points of $A$, which is, in
 many nice cases, large enough to be a boundary.  Then, by imposing some additional assumptions on $A$ and
 $B$,  we give a similar result for certain  non strictly convex Banach
 spaces.  We should note that our method  is based on studying maximal
convex subsets of the unit spheres of $A$ and $B$.  This method
initially emerged in the works of Eilenberg \cite{eil} and Myers
\cite{may} and has later been adapted to the scalar valued case by
Roberts and Lee \cite{lee}.
\section{Preliminaries}
Throughout this paper $\BK$ stands for the scalar fields $\BR$ or
$\BC$. For a compact Hausdorff space $X$ and a Banach space $E$
over $\BK$, $C(X,E)$ is the Banach space of all continuous
$E$-valued functions on $X$ endowed with the supremum norm
$\|\cdot\|_\infty$.  For each $u \in E$, the constant map  $c_u: X
\longrightarrow E$ is defined by $c_u(x)=u$ for each $x \in X$. We
say that a subspace $A$ of $C(X,E)$ is {\em $E$-separating}  if
for any  distinct points $x,x'\in X$ and arbitrary $u \in E$ there
exists $f\in A$ with $f(x)=0 $, $f(x')=u$ and
$\|f\|_\infty=\|u\|$. For every $f \in A$ we put $M(f)=\{x \in X :
\|f(x)\| = \|f\|_\infty\}$.

For a normed space $\E$ we denote the unit sphere of $\E$ by
$S(\E)$ and we put \[ \widetilde{S}(\E)=\{ K: K {\rm \; is \; a\;
maximal\; convex\; subset\;  of\;} S(\E)\}.\] Clearly $S(\E)$ is
not convex and each convex subset of $S(\E)$ is contained in an
element of $\widetilde{S}(\E)$. We note that if  $\E$ is strictly
convex, then all maximal convex  subsets of $S(\E)$ are singleton.

Let $X$ be a compact Hausdorff space and $E$ be a Banach space
over $\Bbb K$. For a subspace $A$ of $C(X,E)$  the {\em Choquet
boundary} of $A$ is denoted by $\Ch(A)$. We recall that $\Ch(A)$ consists of all points  $x \in X$
such that $\nu^* \circ \delta_x $ is an extreme point of the closed unit
ball of $A^{*}$ for some extreme point $\nu^*$ of the closed unit ball of
$E^*$. It is well known that $\Ch(A)$ is  a boundary for $A$ in
the sense that for each $f\in A$ there exists $x\in \Ch(A)$ such
that $\|f(x)\|=\|f\|_\infty$.  For $x \in X$
and $K \subseteq S(E)$ we set
\[V^{A}_{x}=\{ f \in
S(A) : \|f(x)\| =1 \},\,\,\,\, V^{A}_{x, K}=\{ f \in S(A) : f(x)
\in K\}.\]

\begin{lem}\label{conv}
 Let $X$ be a compact Hausdorff space, $E$ be a Banach space  over $\Bbb K$, and $A$ be a $\BK$-subspace of $C(X,E)$.
 Then for each  convex subset $C$ of $S(A)$ there exist
 $x \in X$ and   $K \in \widetilde{S}(E)$ such  that $C \subseteq V^{A}_{x, K}$.
 In particular every maximal convex subset of $S(A)$ is of the form $V^{A}_{x, K}$
 for some $x \in X$ and  $K \in \widetilde{S}(E)$.
\end{lem}
\begin{proof}
Since $C$ is convex, it follows easily that the family $\{M(f): f\in C\}$ of compact subsets of $X$ has finite intersection property and consequently
$\cap_{f\in C} M(f)\neq \emptyset$. Let $x\in \cap_{f\in C} M(f)$ and put $K_0=\{ f(x): f\in C\}$. Then $K_0$ is a convex subset of $S(E)$ and so there exists
$K\in \widetilde{S}(E)$ such that $K_0\subseteq K$. This clearly implies that $C\subseteq V^A_{x,K}$, as desired.
\end{proof}

Let $X$ be a compact Hausdorff space and $A$ be a  $\BK$-subspace
of $C(X,E)$. We call a  point $x\in X$  a {\em strong boundary
point} of $A$ if for each neighborhood $U$ of $x$, $\epsilon>0$,
and  $u\in S(E)$ there exists a function $f \in A$ such that
$\|f\|_\infty=1$, $f(x)=u$ and $\|f(y)\| < \epsilon $ for all
$y\in X\backslash U$. We denote the set of all strong boundary
points of $A$   by $\Theta(A)$. We also denote the set of points
$x\in X$ satisfying the above condition for $\epsilon=1$ by
$\tau(A)$. Hence $\Theta(A)\subseteq \tau(A)$.

\section{Main results}
We begin this section by introducing certain type of points satisfying some maximal convexity conditions.
\begin{defn} {\rm  Let $X$ be a compact Haudsorff space, $E$ be a Banach space over $\Bbb K$, and $A$ be
a $\Bbb K$-subspace of $C(X,E)$. We say that a point $x\in X$ is
of {\em type one }  for $A$ if for each $K \in \widetilde{S}(E)$,
$V^{A}_{x, K}$ is a maximal convex subset of $S(A)$. A point $x\in
X$ is of {\em type two} for $A$ if for each  $K, K' \in
\widetilde{S}(E)$ and $y \in X$ the inclusion  $V^{A}_{x, K}
\subseteq V^{A}_{y, K'}$ implies $x=y$. }
\end{defn}

It is easy to see that  if $A$ is $E$-separating or it contains
constants, then for any point $x\in X$ of type two, the inclusion
$V^{A}_{x, K} \subseteq V^{A}_{y, K'}$, where $K, K' \in
\widetilde{S}(E)$ and $y \in X$,  implies $x=y$ and  $K=K'$.

The set of all type one, respectively  type two points  for $A$
will be denoted by $\eta_1(A)$ and $\eta_2(A)$.

We note that for an arbitrary subspace $A$ of $C(X,E)$  some of
the above defined  sets may be empty. However, as the next lemma
shows, $\eta_1(A)$ and $\eta_2(A)$ contain the set of strong
boundary points of $A$, which is large enough for certain
subspaces $A$.

\begin{lem}\label{1}
Let $X$ be a compact Hausdorff space, $E$ be a Banach space over
$\Bbb K$ and $A$ be a $\Bbb K$-subspace  of $C(X,E)$.  Then
\[\Theta(A) \subseteq \tau(A) \subseteq \eta_2(A)\cap X_0 \subseteq \eta_1(A),\] where
$X_0=\{ x\in X: S(E) \subseteq \{f(x): f\in S(A)\}\}$.
In particular, if $A$ contains constants, then $\eta_2(A)
\subseteq \eta_1(A)$. If $A$ is assumed to be $E$-separating, then
$\eta_{2}(A)=X$.
\end{lem}
\begin{proof}
 The first inclusion is trivial. Take $x \in \tau(A)$ and assume that  $V^{A}_{x, K} \subseteq
V^{A}_{y, K'}$, where $y\in X\backslash \{x\}$
 and $K, K' \in \widetilde{S}(E)$. Fixing $u \in K$ we can find an open neighborhood
$U$ of $x$ and $f \in A$ such that  $y \in X\backslash U$,
$\|f\|_\infty=1$, $f(x)= u$ and $\|f(z)\|<1$ for all $z\in
X\backslash U$. Since $f\in V_{x,K}$ it follows that $f(y)\in K'$
and, in particular, $\|f(y)\|=1$, a contradiction.  This shows
that $x\in \eta_{2}(A)$, that is $\tau(A) \subseteq \eta_2(A)$. Clearly $\tau(A) \subseteq X_0$.

 Now suppose that $x \in \eta_{2}(A)\cap X_0$ and let  $K \in \widetilde{S}(E)$.
 Since  $V^{A}_{x, K}$ is a convex subset of $S(A)$, it is
 contained in a maximal convex subset of $S(A)$. Hence, by Lemma \ref{conv}, there exist
 $y\in X$ and  $K' \in \widetilde{S}(E)$ such that  $V_{y,K}$ is a maximal convex subset of $S(A)$ and $V^{A}_{x, K}
\subseteq V^{A}_{y, K'}$. Therefore, $y=x$ since $x\in \eta_2(A)$.
Hence $V^A_{x,K}\subseteq V^A_{x,K'}$ and, being $x\in X_0$, it
follows easily that $K\subseteq K'$. Thus $K=K'$, that is
$V^A_{x,K}$ is a maximal convex subset of $S(A)$. This concludes
that  $\eta_2(A)\cap X_0\subseteq \eta_1(A)$.

The second part can be easily verified.
\end{proof}

Using a similar argument as in \cite[Lemma 3.2]{jam} we get the
next lemma. We should note that the lemma  is similar to the
additive Bishop's Lemma in scalar case, see for instance
\cite{tonev}.

\begin{lem}\label{Bis}
Let $X$ be a compact Hausdorff space, $E$ be a Banach space and  $A$ be a closed
$\BK$-subspace of $C(X,E)$. Assume that  $x \in \Theta(A)$ and  $f
\in A$ such that $\|f\|_\infty=1$ and $f(x)=0$. Then for each $u \in S(E)$ and $0 <r
<1$ there exists $g \in V^{A}_{x, \{u\}}$ such that $rf+ g \in
V^{A}_{x, \{u\}}$.
\end{lem}

We note that, by the above lemma, for each $x\in \Theta(A)$ and $u\in S(E)$
we have \[\ker(\delta_x) \subseteq s (V_{x,\{u\}})-V_{x,\{u\}})
\] where $s>0$ and $\delta_x: A \longrightarrow E$ is defined by
$\delta_x(f)=f(x)$, $f\in A$. Motivated by this, we say that a
point $x\in X$ is a {\em Bishop point} for $A$ if for each $u\in S(E)$
the above inclusion holds for some $s>0$. We denote the set of
such points for $A$ by $\Omega(A)$. Hence $\Theta(A) \subseteq \Omega(A)$.

In what follows  we assume that $X, Y$ are compact Hausdorff
spaces, $E,F$ are Banach spaces over $\Bbb K$ and $A,B$ are $\Bbb
K$-subspaces of $C(X,E)$ and $C(Y,E)$, respectively. Let $T:A
\longrightarrow B$ be a surjective, not necessarily linear,
isometry. Since, by the Mazur-Ulam theorem, $T-T0$ is real-linear,
without loss of generality we assume that $T0=0$ and  $A$ is
closed in $C(X,E)$. Clearly $T$ maps each maximal convex subset of
$S(A)$ to a such subset of $S(B)$. Hence, by Lemma \ref{conv},
for each $x \in \eta_{1}(A)$ and   $K \in \widetilde{S}(E)$ there
exist $y \in Y$ and $L \in \widetilde{S}(F)$ such that
$T(V^{A}_{x, K})=V^{B}_{y, L}$. In the case that $F$ is strictly
convex, $L$ is a singleton and so there exists  $v\in S(F)$ such
that $T(V^A_{x, \{u\}})\subseteq V^{B}_{y, \{v\}}$ for each $u\in
K$. Motivated by this, for each $x \in \eta_1(A)$ we set
\begin{align*}
H_x= \{ y\in Y:  T(V^A_{x, \{u\}}) \subseteq V^B_{y, \{v\}}\; {\rm
for\; some\;} u\in S(E) \; {\rm and}\; v\in S(F)\}.
\end{align*}
We also put $Y_0= \bigcup_{x\in \Theta(A)} H_x$ and
$Y_1=\bigcup_{x\in \eta_2(A)\cap \Omega(A)\cap X_0} H_x$, where
$X_0$ is as in Lemma \ref{1}. We recall that $X_0=X$ if $A$
contains constants and $\eta_2(A)=X=X_0$ if $A$ is $E$-separating.

We note that if $\Theta(A)=X$ and $\widetilde{S}(F)$ contains a
singleton $\{v\}$ (in particular, if $F$ is strictly convex), then
$Y_0\supseteq \Theta(B)$. Indeed,  for each $y\in \Theta(B)$,
since $T^{-1}$ is also an isometry,  there exist $x\in X$ and
$K\in \widetilde{S}(E)$ such that $T^{-1}(V^B_{y,
\{v\}})=V^A_{x,K}$. This implies that for each $u\in K$ we have
$T(V^A_{x, \{u\}}) \subseteq V^B_{y, \{v\}}$, that is $y\in H_x$.
Hence, in this case $Y_0\supseteq \Theta(B)$.
\begin{lem}\label{dis}
Let $x,x'\in X$ be distinct. In either of cases that
$x,x'\in \Theta(A)$ or $A$ is $E$-separating and $x,x' \in  \Omega(A)$ we have  $H_x\cap H_{x'}=\emptyset$.
\end{lem}
\begin{proof}
We first show that for each  $x \in \eta_2(A)\cap \Omega(A)\cap X_0$ and   $y \in H_x$
if $f\in A$ such that $f(x)=0$, then $Tf(y)=0$. We note that,  by
the definition of $H_x$,  there are  $u \in S(E)$,  $v \in
S(F)$  such that $T(V^{A}_{x, \{u\}}) \subseteq
V^{B}_{y, \{v\}}$. Fixing $r\in (0,1)$, by Lemma \ref{Bis},
there exists $g \in V^{A}_{x, \{u\}}$ such that $rf+ g \in
V^{A}_{x, \{u\}}$. Therefore  $T(rf+g)(y)= v$ and $Tg(y)= v$ which implies, by the real-linearity of $T$,  $Tf(y)=0$.

Now assume that  $x,x'\in \Theta(A)$ are distinct and assume on the
contrary that there exists a point $y$ in $H_x\cap H_{x'}$. Since
$\Theta(A)\subseteq \eta_2(A)\cap \Omega(A)\cap X_0$, it follows from the  above
argument that $Tf(y)=0$ for each $f\in A$ satisfying
either $f(x)=0$ or $f(x')=0$. Let  $u\in  S(E)$ and
$v\in S(F)$ be as  above. Since $x\in \Theta(A)$
there exists $f\in A$ with $\|f\|_\infty=1$, $f(x)=u$ and
$\|f(x')\|<\frac{1}{2}$. Similarly, since $x'\in \Theta(A)$ we can choose
$h\in A$ satisfying $h(x')=f(x')$ and $\|h\|_\infty=\|f(x')\|$. Then  the function $g=f-h$ is an element of $A$ with $g(x')=0$. Hence  $T(g)(y)=0$ and consequently  $v=T(f)(y)=T(h)(y)$. Thus
$1=\|v\|=\|T(f)(y)\|=\|T(h)(y)\|\le
\|T(h)\|_\infty=\|h\|_\infty=\|f(x')\|<\frac{1}{2}$, a contradiction.

Consider the case that $A$ is $E$-separating and $x,x'\in  \Omega(A)$. Let $y\in
H_x\cap H_{x'}$ and $u$ and $v$ be as above. Then, by assumption,  there exists $f
\in A$ such that $\|f\|_\infty=1$, $f(x)=u$ and $f(x')=0$. As before, we get $Tf(y)=0$ while $f\in
V_{x, \{u\}}$ and $T(V_{x,\{u\}}\subseteq V_{y,\{v\}}$, a contradiction.
This shows that in both cases we have $H_x\cap H_{x'}=\emptyset$.
\end{proof}

Using the  above lemma
we can  define a map $\Phi: Y_0 \longrightarrow \Theta(A)$ such
that for each $y\in Y_0$, $\Phi(y)$ is the unique point $x\in
\Theta(A)$ with $y\in H_x$. Clearly $\Phi$ is a well-defined map
which is surjective whenever $F$ is strictly convex. Similarly we
can define a function $\Phi_1: Y_1 \longrightarrow  \Omega(A)$, whenever $A$ is $E$-separating.

\begin{theorem} \label{main}
Let $X,Y$ be compact Hausdorff spaces, $E,F$ be Banach spaces over
$\Bbb K$, where $F$ is strictly convex. Let $A,B$ be
$\BK$-subspaces of $C(X,E)$ and $C(Y,F)$, respectively, and  $T: A
\longrightarrow B$ be a surjective (not necessarily linear)
isometry. Then  there exists a subset $Y_0$ of $Y$, a continuous
surjection $\Phi: Y_0 \longrightarrow \Theta(\overline{A})$  and a
family $\{V_y\}_{y\in Y_0}$  of real-linear operators from $E$ to
$F$ with $\|V_y\| = 1$ such that
\[ Tf(y)=T0(y)+ V_y(f(\Phi(y))) \;\;\;\;\; (f \in A, y \in Y_0).\]
Furthermore,

(i) if $A$ contains constants, then the map $Y_0\longrightarrow B(E,F)$ is continuous with respect to the strong operator topology on $B(E,F)$;

(ii) if $A,B$ contain constants and $T$ maps each constant
function to a  constant function, then all $V_y$ are equal to a real-linear isometry $V:E\longrightarrow F$.
\end{theorem}

\begin{proof} As we noted before, we can assume that $T$ is real-linear and $A$ is closed in $C(X,E)$.
Let $Y_0\subseteq Y$ and $\Phi: Y_0 \longrightarrow \Theta(A)$ be
defined as above.  For each $y \in Y_0$, let $V_y:E
\longrightarrow F$ be defined by $V_y(u)=T(f_0)(y)$, where $f_0\in
A$ satisfies $f_0(\Phi(y))=u$.  We note that there exists a
function $f_0\in A$ satisfying this property, since $\Phi(y) \in
\Theta(A)$.
 Note also that $V_y$ is well defined. Indeed, for $u\in E$
 if  $f_0,f_1\in A$ such that $f_0(\Phi(y))=u=f_1(\Phi(y))$,
 then  $(f_0-f_1)(\Phi(y))=0$ and it follows from real-linearity of $T$ and the argument given in Lemma \ref{dis} that
 $T(f_0)(y)=T(f_1)(y)$. It is easy to see
that  $V_y$ is a real-linear operator and since $f_0\in A$
with $f_0(\Phi(y))=u$  can be chosen such that
$\|f_0\|_\infty=\|u\|$ we have $\|V_y\|\le 1$. Clearly
$Tf(y)=V_y(f(\Phi(y)))$ holds for all $f\in A$ and $y\in Y_0$.

The strict convexity  of $F$ shows that for each $x\in \Theta(A)$,
$H_x$ is nonempty. Hence  $\Phi$ is surjective. To show that
$\Phi$ is continuous, let $y_0\in Y_0$ and $U$ be a neighborhood
of $\Phi(y_0)$ in $\Theta(A)$. Choose an open neighborhood
$\widetilde{U}$ in $X$ with $U=\widetilde{U} \cap \Theta(A)$. By
the definition of $\Phi$, there exist $u\in S(E)$ and $v\in S(F)$
such that $T(V^{A}_{\Phi(y_0), \{u\}}) \subseteq V^{B}_{y_0,
\{v\}}$. Since $\Phi(y_0) \in \Theta(A)$, we can find $f\in
V^{A}_{\Phi(y_0), \{u\}}$  such that $\|f(z)\|\le \frac{1}{2}$ on
$X\backslash \widetilde{U}$. Then $W=\{y\in Y_0:
\|Tf(y)\|>\frac{1}{2}\}$ is a neighborhood of $y_0$ in $Y_0$ and
for each $y\in W$, $\|f(\Phi(y))\| \geq \|V_{y}( f(\Phi(y))\|
=\|Tf(y)\|>\frac{1}{2}$, that is $\Phi(W)\subseteq
\widetilde{U}\cap \Theta(A)$ and so $\Phi$ is continuous.

We now show that for each $y\in Y_0$, $\|V_y\|=1$. Let $y_0 \in
Y_0$ and choose $u\in S(E)$ and $v\in S(F)$ as above. Let $f_0\in
A$ such that $f_0(\Phi(y_0))=u$ and $\|f_0\|_\infty =1$. Then
$V_{y_0}(u)=T(f_0)(y_0)$ and since  $f_0 \in V^{A}_{\Phi(y_0),
\{u\}}$ we have  $T(f_0)(y_0)= v$. Hence $\|V_{y_0}(u)\| =
\|T(f_0)(y_0)\| = 1= \|u\|$ and consequently $\|V_{y_0}\|=1$.

To prove (i) assume that $A$ contains constants. Then for each
$y\in Y_0$ and $u\in E$ we have $V_{y}(u)=T(c_u)(y)$. Hence for
each net $\{y_\alpha\}$ in $Y_0$ converging to a point $y\in Y_0$,
it follows from continuity of $T(c_u)$ that $V_{y_\alpha}(u)\to
V_{y}(u)$, as desired.

Finally to prove (ii) assume that  $A, B$ contain constants and
$T$ maps constants to constants. For each $u\in S(E)$ we have
$\|V_y(u)\|=\|T(c_u)(y)\|=\|v\|=1$ where $v\in S(F)$ such that
$c_v=T(c_u)$. Hence all $V_y$'s are real-isometries  and equal.
\end{proof}
\begin{remark}{\rm

(i) In the above theorem, if $\BK=\BC$ and $T$ is assumed to be
complex linear, then each $V_y$ is also complex-linear.

(ii) If $A$ is assumed to be $E$-separating, then the same
argument can be applied to get a similar  description of $T$  for
all points $y\in Y_1$ and the previously defined map $\Phi_1: Y_1
\longrightarrow \Omega(A)$.

(iii) If $T$ maps constants onto constants, then it is easy to see
that the  real-linear isometry $V:E \longrightarrow F $ is
surjective.}
\end{remark}

For the application of the results we give next corollaries.

As we noted before, the following (S3) condition has been
considered in \cite{kawa1} in some results.

(S3) For each $x \in \Ch(A)$, for each neighborhood $U$ of $x$ and
for each $u \in E$ there exists a $f \in A$ such that $\|f\|_\infty =
\|u\|$, $f (x) = u$ and $f =0$ on $X\backslash U$.

Clearly if (S3) holds for $A$, then we have $\Ch(A) \subseteq
\Theta(A)$ and consequently $\Theta(A)$ is a boundary for $A$. Now
in the next corollary we consider this later condition. Hence this corollary may be compared with Theorem   \ref{kawa}.

\begin{cor}
Let $X, Y$ be compact Hausdorff spaces, $E,
F$ be Banach spaces over $\BK$, where $F$ is strictly convex. Let
 $A$ and $B$ be $\BK$-subspaces of $C(X,E)$ and $C(Y,F)$ such that  $\Ch(A)
\subseteq \Theta(A)$. Then for any surjective isometry $T: A
\longrightarrow B$  there exist a subset $Z$ of $Y$, a continuous
surjection  $\phi: Z \longrightarrow \Ch(A)$ and a family
$\{V_y\}_{y\in Z}$ of real-linear operators  from $E$ to $F$ with
$\|V_y\|=1$ such that
\[ Tf(y)=T0(y)+ V_y(f(\phi(y)))\,\,\,\,\,\,\,\, (f \in A, y \in Z).\]
Moreover, in the case that $A,B$ contain constants,  and $T$ maps constants to constants, all $V_y$'s
are equal to a real-linear isometry $V:E \longrightarrow F$ and $Z$ is a boundary for $B$.
\end{cor}
\begin{proof}
The first part is immediate from Theorem \ref{main}. It suffices
to consider  $Z=\Phi^{-1}(\Ch(A))$ and $\phi=\Phi |_Z$. For the
second part, assume that $A,B$ contain constants and $T$ sends
constants to constants. By the above theorem, there exists a
real-linear isometry $V:E \longrightarrow F$ such that $V_y=V$ for
all $y\in Y_0$. To show that $Z$ is a boundary for $B$, let $g\in
B$ and $f\in A$ such that $Tf-T0=g$. Since $\phi: Z
\longrightarrow \Ch(A)$ is surjective and $\Ch(A)$ is a boundary
for $A$, there exists a point $y_0\in Z$ such that $\sup_{y\in
Z}\|f(\phi(y))\|=\|f(\phi(y_0))\|=\|f\|_\infty$. Thus
\begin{align*}
 \|f\|_\infty &=\|g\|_\infty \ge \|g(y_0)\|=\|V(f(\phi(y_0)))\| \\
              &= \|f(\phi(y_0))\|=\|f\|_\infty.
\end{align*}
Therefore, $\|g\|_\infty=\|g(y_0)\|$, that is $Z$ is a boundary
for $B$.
\end{proof}

We recall that a  subspace $A$ of $C(X,E)$ is called {\em
completely regular} if for each  $x\in X$, $u\in S(E)$ and closed
subset $F$ of $X$ not containing $x$, there exists $f \in A$ with
$f(x)=u$, $\|f\|_\infty=1$ and $f(z)=0$ for each $z \in F$.
Obviously for such subspaces we have $ \Theta(A)=X$.  So we get
the following generalization of Cambern's result \cite{cam}, which
is also a generalization of \cite[Theorem 1]{font1} for not
necessarily linear isometries.
\begin{cor}
Let $X,Y$ be  compact Hausdorff spaces and  $E, F$ be Banach
spaces over $\BK$, where $F$ is strictly convex. Let $A$ be a
completely regular $\BK$-subspace of $C(X,E)$ and  $B$ be a
$\BK$-subspace of $C(Y,F)$. Then for any surjective isometry  $T:
A \longrightarrow B$  there exist a subset $Y_0$ of $Y$, a
continuous surjection $\Phi: Y_0 \longrightarrow X$  and a
collection $\{V_y\}_{y\in Y_0}$ of real-linear operators from $E$
to $F$ with $\|V_y\| = 1$ such that
\[ Tf(y)=V_y(f(\Phi(y))) \,\,\,\,\,\, (f \in A, y \in Y_0).\]
 \end{cor}

In the next theorem we give a similar result for surjective
isometries, in not necessarily strictly convex case. We consider
the case that $F$ is a Banach space whose unit sphere $S(F)$ has
at least a point $v$ such that $\{v\}$ is a maximal convex subset
of $S(F)$.  Before stating our result we give an example of such
(non strictly convex) Banach spaces.

\begin{exam}{\rm
Let $n\in \BN$ and $K$ be a compact symmetric convex subset of
$\BR^n$ with nonempty interior. Then we set $\|0\|=0$ and for each
nonzero point $x\in \BR^n$ we define $\|x\|= \frac{1}{\max\{t\in
\BR: tx \in K\}}$. Then $\|\cdot \|$ defines a norm on $\BR^n$
whose closed unit ball is $K$. In particular, consider the
following subset of $\BR^2$:
\begin{align*}
K&=\{ ({\rm sec}(\theta),\theta): \theta\in (0,\frac{\pi}{4} )\cup
(\frac{7\pi}{4}, 2\pi)\} \bigcup \{(\sqrt{2},\theta): \theta\in
(\frac{\pi}{4}, \frac{3 \pi}{4}) \cup (\frac{5\pi}{4}, \frac{7
\pi}{4})\} \\
&\bigcup \{ (-{\rm sec}(\theta), \theta): \theta \in
(\frac{3\pi}{4}, \frac{5 \pi}{4})\}.
\end{align*}
 Then $K$ satisfies the above mentioned
properties and so $(\BR^2, \|\cdot \|)$ is a Banach space with
closed unit ball $K$. It is clear that this Banach space is not
strictly convex, and there are infinitely many points in $K$ which
are maximal convex subsets of $K$.}
\end{exam}

\begin{theorem}
Let $X,Y$ be compact Hausdorff spaces, $E,F$ be Banach spaces such
that $\widetilde{S}(F)$ contains at least one singleton and let
$A$, $B$ be $\BK$-subspaces of $C(X,E)$ and $C(Y,F)$ with
$\Theta(A)=X$ and $\theta(B)=Y$.  Then for any surjective isometry
$T: A\longrightarrow B$, there exist a continuous map
$\Phi:Y\longrightarrow X$, a family $\{V_y\}_{y\in Y}$ of linear
operators from $E$ to $F$ with $\|V_y\|\le 1$, for all $y\in Y$
such that
\[ Tf(y)= T0(y)+V_y(f(\Phi(y)))\,\,\,\,\,\,\,(f\in A, y\in Y).\]
If, in addition,  $\widetilde{S}(E)$ also  contains a singleton,
then $\Phi$ is a homeomorphism and all $V_y$ are isometries.
\end{theorem}
\begin{proof}
As before we may assume that $T$ is real-linear. By hypothesis,
there exists $v\in S(F)$, such that $\{v\}$ is a maximal convex
subset of $S(F)$. For each $y\in Y$, since $Y=\Theta(B)$, it
follows from Lemma \ref{1} that $V_{y, \{v\}}$ is a maximal convex
subset of $S(B)$. Being $T^{-1}$ an isometry, there exists $x\in
X$ and $K\in \widetilde{S}(E)$ such that
$T^{-1}(V_{y,\{v\}})=V_{x,  K}$, that is $T(V_{x,K})=V_{y,\{v\}}$.
We note that the point $x\in X$ satisfying the above equality for
some $K\in \widetilde{S}(E)$ is unique. Indeed, if
$V_{x,K}=V_{z,L}$ where $z\in X$ is distinct from $x$ and $L\in
\widetilde{S}(E)$, then since $x$ and $z$ are strong boundary
points for $A$ we can find easily a function $f\in V_{x,K}$ with
$\|f(z)\|< \frac{1}{2}$, a contradiction. The same argument as in
Lemma \ref{dis} shows that for all $f\in A$, $f(x)=0$ implies
$Tf(y)=0$ and consequently for each $f,h\in A$ with $f(x)=h(x)$ we
have $Tf(y)=Th(y)$. Thus we can define a real linear  operator
$V_y:E\longrightarrow F$ by $V_y(e)=Tf(y)$ where $f\in A$ such
that $f(x)=e$.  Since $X=\Theta(A)$, the above function $f\in A$
can be chosen such that $\|f\|_\infty=\|e\|$ and $f(x)=e$. This
shows that $\|V_y\|\le 1$. Clearly $Tf(y)=V_y(f(x))$ holds for all
$f\in A$. By the above argument we can define a map $\Phi:Y
\longrightarrow X$ and a family of real linear operators
$\{V_y\}_{y\in Y}$ such that
\[Tf(y)=V_y(f(\Phi(y)))\;\;\;\;\;\;\; (f\in A, y\in Y).\]

As in Theorem \ref{main} we see that $\Phi:Y\longrightarrow X$ is
continuous.

For the second part, assume that $\widetilde{S}(E)$ also contains
a singleton. Then using the above discussion for $T^{-1}$ we can
define a continuous map $\Psi: X\longrightarrow Y$ and a family
$\{W_x\}_{x\in X}$ of real-linear operators from $F$ to $E$ such
that $\|W_x\|\le 1$ and
\[ T^{-1}(g)(x)=W_x(g(\Psi(x)))\;\;\;\;\;\;\;(g\in B, x\in X). \]
Thus, for each $f\in A$ and $x\in X$ we have
\[f(x)=W_x(Tf(\Psi(x)))=W_x(V_{\Psi(x)}(f(\Phi(\Psi(x))))\]
If $x\in X$ and $\Phi(\Psi(x)))\neq x$, then there exists $f\in A$
with $\|f(x)\|=1$ and $\|f(\Phi(\Psi(x)))\|\le \frac{1}{2}$. Hence
\[1=\|f(x)\|=\|W_x(V_x(f(\Phi(\Psi(x))))\|\le
\|f(\Phi(\Psi(x)))\|\le \frac{1}{2}\] which is a impossible.
Therefore, $\Phi(\Psi(x)))=x$ for all $x\in X$. Similar argument
shows that $\Psi(\Phi(y)))=y$ for all $y\in Y$, that is
$\Psi=\Phi^{-1}$, in particular, $\Phi$ is a homeomorphism. By the
above argument we have
\[ f(x)=W_x(V_{\Psi(x)}(f(x))\,\,\,\,\,\,\,(f\in A, x\in X). \]
Since for each $x\in X$ and $e\in E$ we can choose $f\in A$ with
$f(x)=e$, it follows from the above equality that
$W_x(V_{\Psi(x)}(e))=e$ for all $e\in E$. Similarly,
$V_y(W_{\Phi(y)}(e'))=e'$ for all $e'\in F$.  Hence
$V_y=W_{\Phi(y)}^{-1}$ and $\|V_y\|=\|W_{\Phi(y)}\|=1$, that is
each $V_y$ is an isometry, as desired.
\end{proof}

\vspace*{.25cm} {\bf Acknowledgement.} The first and the second
authors were partially supported by Iran National Science
Foundation: INSF (Grant No. 95002593)

\vspace*{.25 cm}

\address{{Arya Jamshidi,  Department of Pure Mathematics, Faculty of  Mathematical
Sciences,  Tarbiat Modares University, Tehran, 14115-134, Iran}}

\email{{\em E-mail address:} arya.Jamshidi@modares.ac.ir }

\vspace*{.25 cm}

\address{{Fereshteh Sady,  Department of Pure Mathematics, Faculty of  Mathematical
Sciences,  Tarbiat Modares University, Tehran, 14115-134, Iran}}

\email{{\em E-mail adress:}  sady@modares.ac.ir}

\begin{thebibliography}{99}



\bibitem{Al} H. Al-Halees, R. Fleming, {\em Extreme point methods and Banach--Stone theorem}, J. Aust. Math. Soc. 75 (2003), 125--143.

\bibitem{Bot-E} F. Botelho, {\em Erratum to: Surjective isometries on spaces of
vector-valued continuous and Lipschitz functions Fernanda
Botelho}, Positivity 20 (2016), 757--759.

\bibitem{bote} F. Botelho, J. Jamison, {\em Surjective isometries on spaces of vector-valued continuous and Lipschitz functions}, Positivity 17 (2013), 395--405.


\bibitem{cam} M. Cambern, {\em A Holszty\'nski theorem for spaces of continuous vector-valued functions},  Studia Math. 63 (1978), 213--217.

\bibitem{eil} S. Eilenberg, {\em Banach space methods in topology},  Ann. Math., Vol. 43 (1942),  568--579.

\bibitem{Flem1} R. Fleming, J. Jamison, {\em Isometries in Banach Spaces: Function Spaces}, Volume one,  CRC Press, 2003.


\bibitem{Flem2} R. Fleming, J. Jamison, {\em Isometries in Banach Spaces: Vector-valued Function Spaces and Operator Spaces}, Volume Two, CRC Press, 2007.

\bibitem{font1} J. Font,  {\em Linear isometries between certain subspaces of continuous vector-valued functions},  Illinois J. Math. 42 (1998), 389--397.

\bibitem{jam} A. Jamshidi,  F. Sady,   {\em Extremely strong boundary points and real-linear isometries}, Tokyo J. Math., 38 (2015,) 477--490.

\bibitem{jerr} M. Jerison, {\em The space of bounded maps into a Banach space}, Ann. Math. 52 (1950), 309--327.


\bibitem{kawa1} K. Kawamura,  {\em Linear surjective isometries between vector-valued function spaces},  J.  Aust.  Math.  Soc. 100  (2016), 349--373.

\bibitem{miura1} H. Koshimizu, T. Miura, H. Takagi and S.-E. Takahasi, {\em Real-linear isometries between subspaces of continuous functions}, J. Math. Anal. Appl. 413 (2014), 229--241.

\bibitem{miura2} T. Miura, {\em Real-linear isometries between function algebras}, Cent. Eur. J. Math. 9 (2011), 778--788.

\bibitem{may} S. B. Myers, {\em Banach spaces of continuous functions},  Ann. Math.,  49 (1948),  132--140.

\bibitem{lee} K. Roberts, K. Lee, {\em Nonlinear isometries between function spaces}, Ann. Funct. Anal. (2017),  to appear.

\bibitem{tonev} T. Tonev and R. Yates, {\em Norm-linear and norm-additive operators between uniform algebras}, J. Math. Anal. Appl. 57 (2009), 45--53.

\end{thebibliography}
\end{document}